\documentclass[secthm,seceqn,amsthm,ussrhead,12pt]{amsart}
\usepackage{amsmath,latexsym}
\usepackage[english]{babel}
\usepackage[psamsfonts]{amssymb}
\usepackage{times}
\usepackage{cite}

\usepackage[mathcal]{euscript}
\numberwithin{equation}{section} \textwidth=17.5cm
\newtheorem{theorem}{Theorem}[section]
\newtheorem{lemma}[theorem]{Lemma}

\numberwithin{equation}{section}

\begin{document}

\title[Derivations on algebras of measurable operators]{Innerness of continuous derivations
on algebras of measurable operators affiliated with finite von
Neumann algebras}

\author{Shavkat Ayupov and Karimbergen Kudaybergenov}

\address[Shavkat Ayupov]{Institute of
 Mathematics National University of
Uzbekistan,
 100125  Tashkent,   Uzbekistan
 and
 the Abdus Salam International Centre
 for Theoretical Physics (ICTP),
  Trieste, Italy}
\email{sh$_{-}$ayupov@mail.ru}

\address[Karimbergen Kudaybergenov]{Department of Mathematics, Karakalpak state
university, Nukus 230113, Uzbekistan.} \email{karim2006@mail.ru}

\maketitle
\begin{abstract}
This paper is devoted to derivations on the algebra $S(M)$ of all
measurable operators affiliated with a finite von Neumann algebra
$M.$  We prove that if  $M$ is  a finite von Neumann algebra with
a faithful normal semi-finite trace $\tau$, equipped with the
locally measure topology  $t,$ then every $t$-continuous
derivation $D:S(M)\rightarrow S(M)$ is inner. A similar result is valid
for derivation on the algebra $S(M,\tau)$ of $\tau$-measurable operators
equipped with the measure topology $t_{\tau}$.
\end{abstract} \maketitle

\bigskip

\section{Introduction}

\medskip

Given an algebra $\mathcal{A},$ a linear operator
$D:\mathcal{A}\rightarrow \mathcal{A}$ is called a
\textit{derivation}, if $D(xy)=D(x)y+xD(y)$ for all $x, y\in
\mathcal{A}$ (the Leibniz rule). Each element $a\in \mathcal{A}$
implements a derivation $D_a$ on $\mathcal{A}$ defined as
$D_a(x)=[a, x]=ax-xa,$ $x\in \mathcal{A}.$ Such derivations $D_a$
are said to be \textit{inner derivations}. If the element $a,$
implementing the derivation $D_a,$ belongs to a larger algebra
$\mathcal{B}$ containing $\mathcal{A},$ then $D_a$ is called
\textit{a spatial derivation} on $\mathcal{A}.$

One of the main problems in the theory of derivations is to prove
the automatic continuity, ``innerness'' or ``spatiality'' of
derivations, or to show the existence of non-inner and
discontinuous derivations on various topological algebras. In
particular, it is a general algebraic problem to find algebras
which admit only inner derivations. Examples of algebra for which
any derivation is inner include:

\begin{itemize}
\item finite dimensional simple central algebras (see \cite[p.\ 100]{Her});
\item simple unital $C^{\ast}$-algebras (see the main theorem of \cite{Sak68});
\item the algebras $B(X),$ where $X$ is a Banach
space (see \cite[Corollary 3.4]{Che}).
\item von Neumann algebras (see \cite[Theorem 1]{Sak66})
\end{itemize}

\medskip

A related problem is:
\begin{quotation}
Given an algebra $\mathcal{A},$ is there an algebra $\mathcal{B}$
containing $\mathcal{A}$ as a subalgebra such that any derivation
of $\mathcal{B}$ is inner and any derivation of the algebra
$\mathcal{A}$ is spatial in $\mathcal{B}$?
\end{quotation}

\medskip

The following are some examples for which the answer is positive:

\begin{itemize}
\item $C^{\ast}$-algebras (see \cite[Theorem 4]{Kad} or \cite[Theorem 2]{Sak66});
\item standard operator algebras on a Banach space $X$, i.e.\
subalgebras of $B(X)$ containing all finite rank operators (see
\cite[Corollary~3.4]{Che}).
\end{itemize}

\medskip

In \cite{Alb2} and \cite{AK1}, derivations on various subalgebras
of the algebra $LS(M)$ of locally measurable operators with
respect to  a von Neumann algebra $M$ has been considered. A
complete description of derivations has been obtained in the case
when $M$ is of type I and III. Derivations on  algebras of
measurable and locally measurable operators, including rather non
trivial commutative case, have been studied  by many authors
\cite{Alb2, AK2, AK1, Ber, BPS, Ber2, Ber3, Ber4}. A comprehensive
survey of recent results concerning derivations on various
algebras of unbounded operators affiliated with von Neumann
algebras can be found in \cite{AK2}.

If we consider the algebra $S(M)$ of all measurable
operators affiliated with  a type III von Neumann algebra $M$, then it is clear that
$S(M)=M$. Therefore from the results of  \cite{Alb2} it follows that
 for type I$_\infty$ and type III von Neumann algebras $M$ every derivation on $S(M)$ is
automatically inner and, in particular, is continuous in the local
measure topology.
 The problem of description of the structure of
derivations in the case of type II algebras has been open so far
and seems to be rather difficult.

In this connection several open problems concerning innerness and
automatic continuity of derivations on the algebras $S(M)$ and
$LS(M)$ for type II von Neumann algebras have been posed in
\cite{AK2}. First positive results in this direction were recently
obtained in \cite{Ber2, Ber4}, where automatic continuity has been
proved for derivations on
 algebras of  $\tau$-measurable and locally
 measurable operators affiliated with properly infinite von Neumann algebras.

 Another problem  in \cite[Problem 3]{AK2} asks the following
question:

Let $M$ be a type II von Neumann algebra with a faithful normal
semi-finite trace $\tau.$ Consider  the algebra $S(M)$  (respectively $LS(M)$) of all
measurable (respectively locally measurable) operators affiliated with $M$ and equipped with the
locally measure topology  $t.$ Is every  $t$-continuous derivation
$D:S(M)\rightarrow S(M)$ (respectively, $D:LS(M)\rightarrow LS(M)$) necessarily inner?

In the present paper we suggest a solution of this problem  for
 type II$_1$  von Neumann algebras (in this case $LS(M)=S(M)$).  Namely, we
prove that if  $M$ is  a finite von Neumann algebra and
$D:S(M)\rightarrow S(M)$ is a $t$-continuous derivation then $D$
is inner. A similar result is proved
for derivation on the algebra $S(M,\tau)$ of all $\tau$-measurable operators
equipped with the measure topology $t_{\tau}$.

\section{Algebras of measurable operators}

Let  $B(H)$ be the $\ast$-algebra of all bounded linear operators
on a Hilbert space $H,$ and let $\textbf{1}$ be the identity
operator on $H.$ Consider a von Neumann algebra $M\subset B(H)$
with the operator norm $\|\cdot\|$ and  with a faithful normal
semi-finite trace $\tau.$ Denote by $P(M)=\{p\in M:
p=p^2=p^\ast\}$ the lattice of all projections in $M.$

A linear subspace  $\mathcal{D}$ in  $H$ is said to be
\emph{affiliated} with  $M$ (denoted as  $\mathcal{D}\eta M$), if
$u(\mathcal{D})\subset \mathcal{D}$ for every unitary  $u$ from
the commutant
$$M'=\{y\in B(H):xy=yx, \,\forall x\in M\}$$ of the von Neumann algebra $M.$

A linear operator  $x: \mathcal{D}(x)\rightarrow H,$ where the
domain  $\mathcal{D}(x)$ of $x$ is a linear subspace of $H,$ is
said to be \textit{affiliated} with  $M$ (denoted as  $x\eta M$)
if $\mathcal{D}(x)\eta M$ and $u(x(\xi))=x(u(\xi))$
 for all  $\xi\in
\mathcal{D}(x)$  and for every unitary  $u\in M'.$

A linear subspace $\mathcal{D}$ in $H$ is said to be
\textit{strongly dense} in  $H$ with respect to the von Neumann
algebra  $M,$ if

1) $\mathcal{D}\eta M;$

2) there exists a sequence of projections $\{p_n\}_{n=1}^{\infty}$
in $P(M)$  such that $p_n\uparrow\textbf{1},$ $p_n(H)\subset
\mathcal{D}$ and $p^{\perp}_n=\textbf{1}-p_n$ is finite in  $M$
for all $n\in\mathbb{N}$.

A closed linear operator  $x$ acting in the Hilbert space $H$ is
said to be \textit{measurable} with respect to the von Neumann
algebra  $M,$ if
 $x\eta M$ and $\mathcal{D}(x)$ is strongly dense in  $H.$

 Denote by $S(M)$  the set of all linear operators on $H,$ measurable with
respect to the von Neumann algebra $M.$ If $x\in S(M),$
$\lambda\in\mathbb{C},$ where $\mathbb{C}$  is the field of
complex numbers, then $\lambda x\in S(M)$  and the operator
$x^\ast,$  adjoint to $x,$  is also measurable with respect to $M$
(see \cite{Seg}). Moreover, if $x, y \in S(M),$ then the operators
$x+y$  and $xy$  are defined on dense subspaces and admit closures
that are called, correspondingly, the strong sum and the strong
product of the operators $x$  and $y,$  and are denoted by
$x\stackrel{.}+y$ and $x \ast y.$ It was shown in \cite{Seg} that
$x\stackrel{.}+y$ and $x \ast y$ belong to $S(M)$ and these
algebraic operations make $S(M)$ a $\ast$-algebra with the
identity $\textbf{1}$  over the field $\mathbb{C}.$ Here, $M$ is a
$\ast$-subalgebra of $S(M).$ In what follows, the strong sum and
the strong product of operators $x$ and $y$  will be denoted in
the same way as the usual operations, by $x+y$  and $x y.$

It is clear that if the von Neumann algebra $M$ is finite then every linear operator
affiliated with $M$ is measurable and, in particular, a self-adjoint operator is
measurable with respect to $M$ if and only if all its
 spectral projections belong to $M$.

 Let   $\tau$ be a faithful normal semi-finite trace on
 $M.$ We recall that a closed linear operator
  $x$ is said to be  $\tau$\textit{-measurable} with respect to the von Neumann algebra
   $M,$ if  $x\eta M$ and   $\mathcal{D}(x)$ is
  $\tau$-dense in  $H,$ i.e. $\mathcal{D}(x)\eta M$ and given   $\varepsilon>0$
  there exists a projection   $p\in M$ such that   $p(H)\subset\mathcal{D}(x)$
  and $\tau(p^{\perp})<\varepsilon.$
   Denote by  $S(M,\tau)$ the set of all   $\tau$-measurable operators affiliated with $M.$

Note that if the trace $\tau$ is finite then $S(M,\tau)=S(M).$

    Consider the topology  $t_{\tau}$ of convergence in measure or \textit{measure topology}
    on $S(M, \tau),$ which is defined by
 the following neighborhoods of zero:
$$V(\varepsilon, \delta)=\{x\in S(M, \tau): \exists\, e\in P(M),
 \tau(e^{\perp})<\delta, xe\in
M,  \|xe\|<\varepsilon\},$$  where $\varepsilon, \delta$ are
positive numbers.

 It is well-known \cite{Nel} that $M$ is $t_\tau$-dense in $S(M, \tau)$
and $S(M, \tau)$ equipped with the measure topology is a complete
metrizable topological $\ast$-algebra.

Let $M$ be a finite  von Neumann algebra with a faithful normal
semi-finite trace $\tau.$ Then there exists a family
$\{z_i\}_{i\in I}$ of mutually orthogonal central projections in
$M$ with $\bigvee\limits_{i\in I}z_i=\mathbf{1}$ and such that
$\tau(z_i)<+\infty$  for every $i\in I$  (such family exists
because $M$ is a finite algebra). Then the algebra  $S(M)$ is
$\ast$-isomorphic to the algebra $\prod\limits_{i\in I}S(z_iM)$
(with the coordinate-wise operations and involution), i.e.
$$
S(M)\cong\prod\limits_{i\in I}S(z_iM)
$$ ($\cong$ denoting
$\ast$-isomorphism of algebras) (see \cite{Mur}).

This property implies that given any family $\{z_i\}_{i\in I}$ of
mutually orthogonal central projections in $M$ with
$\bigvee\limits_{i\in I}z_i=\textbf{1}$ and a  family of elements
$\{x_i\}_{i\in I}$ in $S(M)$ there exists a unique element $x\in
S(M)$ such that $z_i x=z_i x_i$ for all $i\in I.$

Let $t_{\tau_i}$  be the measure topology on $S(z_iM)=S(z_iM,
\tau_i),$ where $\tau_i=\tau|_{z_iM},$ $i\in I.$ On the algebra
$S(M)\cong\prod\limits_{i\in I}S(z_iM)$ we  consider the topology
$t$ which is the Tychonoff product of the topologies $t_{\tau_i},
i\in I.$ This topology coincides with so-called \textit{locally
measure topology} on $S(M)$ (see \cite[Remark 2.7]{Ber4}).

 It is known \cite{Mur} that  $S(M)$ equipped with the locally measure topology is a
 topological $\ast$-algebra.
Note that if the trace $\tau$ is finite then $t=t_{\tau}.$

\section{The Main results}

\medskip

Given a von Neumann algebra   $M$ with a faithful normal finite
trace $\tau,$  $\tau(\mathbf{1})=1,$ we consider the $L_2$-norm
$$
\|x\|_2=\sqrt{\tau(x^\ast x)},\, x\in M.
$$

 Denote by $\mathcal{U}(M)$  and $\mathcal{GN}(M)$ the set of all unitaries in
 $M$ and   the set of all partially isometries  in $M,$ respectively.

 A partial ordering can be defined on the set $\mathcal{GN}(M)$
as follows:
$$
u\leq_1 v \Leftrightarrow u u^\ast\leq v v^\ast,\, u= uu^\ast v.
$$
It is clear  that
$$
u\leq_2 v \Leftrightarrow  u^\ast u\leq  v^\ast v,\, u=  v u^\ast
u
$$
is also defined a partial ordering  on the set $\mathcal{GN}(M)$
and
$$
u\leq_1 v \Leftrightarrow   u^\ast \leq_2  v^\ast.
$$

Note that $u^\ast u=r(u)$ is the right support of $u,$ and
$uu^\ast=l(u)$ is the left support of $u.$

The $t$-continuity of algebraic operations on $S(M)$ implies that
every inner derivation on $S(M)$ is $t$-continuous.

The following main result of the paper shows that the converse
implication is also true.

\begin{theorem}\label{main}
Let $M$ be a finite   von Neumann algebra with a faithful normal
semi-finite trace $\tau$. Then every $t$-continuous derivation
$D:S(M)\rightarrow S(M)$ is inner.
\end{theorem}

For the proof of this theorem we need several lemmata.

For $f\in \ker D$ and $x\in S(M)$ we have
$$
D(fx)=D(f)x+fD(x)=fD(x),
$$
i.e.
$$
D(fx)=fD(x).
$$
Likewise
$$
D(xf)=D(x)f.
$$
This simple properties  will be frequently used below.

Let $D$ be a derivation on $S(M).$ Let us define a mapping
$D^\ast:S(M)\rightarrow S(M)$  by setting
$$
D^\ast(x)=(D(x^\ast))^\ast,\, x\in S(M).
$$
 A direct verification
shows that $D^\ast$  is also a derivation on $S(M).$  A derivation
$D$ on $S(M)$  is said to be \textit{skew-hermitian}, if $D^\ast=-D,$ i.e.
$D(x^\ast)=-D(x)^\ast$ for all $x\in S(M).$ Every derivation $D$
on $S(M)$ can be represented in the form $D= D_1+ i D_2,$ where
$$
D_1 = (D -D^\ast)/2,\quad D_2  = (-D - D^\ast)/2i
$$
 are skew-hermitian derivations on $S(M).$

 It is clear that a derivation $D$  is inner  if and only if the skew-hermitian derivations
$D_1$  and $D_2$ are inner.

Therefore further we may assume that $D$ is a skew-hermitian
derivation.

\begin{lemma}\label{duu}  For every $v\in \mathcal{GN}(M)$
the element $vv^\ast D(v)v^\ast$ is hermitian.
\end{lemma}

\begin{proof}
First note that if $p$ is a projection, then $pD(p)p=0$. Indeed, $D(p)=D(p^2)=D(p)p+pD(p)$.
Multiplying this equality by $p$  from both sides we obtain $pD(p)p=2pD(p)p$, i.e. $pD(p)p=0$.

Now take an arbitrary $v\in \mathcal{GN}(M).$ Taking into account that
$vv^\ast v=v$ and  $D$ is skew-hermitian, we get
\begin{eqnarray*}
\left(v v^\ast D(v) v^\ast\right)^\ast  &=& vD(v)^\ast v v^\ast=
-vD(v^\ast) v v^\ast=\\
& = & -vD(v^\ast v) v^\ast+   v v^\ast D(v) v^\ast= -v
\left(v^\ast vD(v^\ast v) v^\ast v\right) v^\ast+
\\
& + &    v v^\ast D(v) v^\ast= v v^\ast D(v) v^\ast,
\end{eqnarray*}
because $v^\ast v$ is a projection and therefore $v^\ast vD(v^\ast
v) v^\ast v=0.$ So
$$
\left(vv^\ast D(v) v^\ast\right)^\ast=vv^\ast D(v) v^\ast.
$$
The proof is complete.
\end{proof}

\begin{lemma}\label{five}  Let $n\in \mathbb{N}$ be a fixed number and
let  $v\in \mathcal{GN}(M)$ be  a partially isometry. Then
$$
vv^\ast D(v)v^\ast\geq n vv^\ast
$$
if and only if
$$
v^\ast vD(v^\ast) v\leq -n v^\ast v.
$$
\end{lemma}

\begin{proof}
Take an arbitrary $v\in \mathcal{GN}(M)$ such that $ vv^\ast
D(v)v^\ast\geq n vv^\ast.$ Multiplying this equality from the left
side by $v^\ast$ and from the right side by $v,$ we obtain
$$
v^\ast vv^\ast D(v)v^\ast v\geq n v^\ast v v^\ast v,
$$
i.e.
$$ v^\ast  D(v)v^\ast v\geq n v^\ast  v.
$$
Since
\begin{eqnarray*}
v^\ast D(v)v^\ast v   &=&  v^\ast D(v v^\ast) v-    v^\ast v
D(v^\ast)v = v^\ast \left(v v^\ast D(v v^\ast) v v^\ast\right) v-
\\
& - &    v^\ast v D(v^\ast)v= -v^\ast v D(v^\ast)v,
\end{eqnarray*}
because $v v^\ast$ is a projection and therefore $v v^\ast D(v
v^\ast) v v^\ast =0.$ So
$$
-v^\ast v D(v^\ast)v\geq n v^\ast v,
$$
i.e.
$$
v^\ast v D(v^\ast)v\leq -n v^\ast v.
$$
In a similar way we can prove the converse implication.
  The proof is complete.
\end{proof}

\begin{lemma}\label{orto}  Let   $v_1\in \mathcal{GN}(M)$ be
a partially isometry  and let  $v_2\in \mathcal{GN}(pMp),$ where
\linebreak  $p=\mathbf{1}- v_1v_1^\ast \vee v_1^\ast v_1 \vee
s(iD(v_1v_1^\ast))\vee s(iD(v_1^\ast v_1))$ and $s(x)$ denotes the
support of a hermitian element $x.$ Then
$$
 (v_1+v_2)(v_1+v_2)^\ast D(v_1+v_2)(v_1+v_2)^\ast=v_1v_1^\ast D(v_1)v_1^\ast
+v_2v_2^\ast D(v_2)v_2^\ast.
$$
\end{lemma}

\begin{proof} Since $v_2\in
\mathcal{GN}(pMp)$ we get
$$
v_1v_2^\ast=v_2^\ast v_1=v_2v_1^\ast=v_1^\ast v_2=0,
$$
$$
v_2^\ast D(v_1v_1^\ast)= D(v_1v_1^\ast)v_2=v_2D(v_1^\ast v_1) =
D(v_1^\ast v_1)v_2^\ast=0.
$$
Thus
$$
v_1v_1^\ast D(v_2)=D(v_1v_1^\ast v_2)-D(v_1v_1^\ast)v_2=0,
$$
$$
v_2^\ast D(v_1)v_2^\ast =D(v_2^\ast v_1) v_2^\ast - D(v_2^\ast)
v_1v_2^\ast=0,
$$
$$
v_1^\ast D(v_1)v_2^\ast =D(v_1^\ast v_1) v_2^\ast - D(v_1^\ast)
v_1v_2^\ast=0,
$$
$$
v_2^\ast D(v_1)v_1^\ast =v_2^\ast D(v_1 v_1^\ast) - v_2^\ast v_1
D(v_1^\ast)=0,
$$
$$
D(v_2)v_1^\ast =D(v_2)v_1^\ast v_1 v_1^\ast=D(v_2v_1^\ast
v_1)v_1^\ast- v_2 D(v_1^\ast v_1) v_1^\ast=0.
$$
 Taking into account these equalities we get
\begin{eqnarray*}
(v_1+v_2)(v_1+v_2)^\ast D(v_1+v_2)(v_1+v_2)^\ast &=&
(v_1v_1^\ast+v_2v_2^\ast) D(v_1+v_2)(v_1+v_2)^\ast=
\\
& = &  v_1v_1^\ast D(v_1)v_1^\ast +
  v_2v_2^\ast D(v_2)v_2^\ast+
\\
& + &  v_1v_1^\ast D(v_2)v_1^\ast +
  v_2v_2^\ast D(v_1)v_2^\ast+
  \\
& + &  v_1v_1^\ast D(v_1)v_2^\ast +
  v_2v_2^\ast D(v_2)v_1^\ast+
  \\
& + &  v_1v_1^\ast D(v_2)v_2^\ast +
  v_2v_2^\ast D(v_1)v_1^\ast=
  \\
& = &  v_1v_1^\ast D(v_1)v_1^\ast +
  v_2v_2^\ast D(v_2)v_2^\ast.
  \end{eqnarray*}
The proof is complete. \end{proof}

Let $p\in M$ be a projection. It is clear that the mapping
$$
pDp:x\rightarrow pD(x)p,\, x\in pS(M)p
$$
is a derivation on $pS(M)p=S(pMp).$

 The
following lemma is  one of the key steps in the  proof of the main
result.

\begin{lemma}\label{boun}  Let $M$ be a
von Neumann algebra  with a faithful normal finite trace $\tau$,
  $\tau(\mathbf{1})=1.$ There exists a sequence of projections
$\{p_n\}$ in $M$ with $\tau(\mathbf{1}-p_n)\rightarrow 0$ such
that the derivation $p_nDp_n$ maps $p_nMp_n$ into itself for all
$n\in \mathbb{N}.$
\end{lemma}

\begin{proof}
For each  $n\in \mathbb{N}$ consider the set
$$
\mathcal{F}_n=\{v\in \mathcal{GN}(M): vv^\ast D(v)v^\ast\geq n
vv^\ast\}.
$$

Note that $0\in \mathcal{F}_n$,  so $\mathcal{F}_n$ is not empty.
Let us show that the set $\mathcal{F}_n$ has a maximal element
with respect to the order $\leq_1.$

Let $\{v_\alpha\}\subset \mathcal{F}_n$ be a totally ordered net.
We will show that $v_\alpha \stackrel{t_\tau}\longrightarrow v$
for some $v \in \mathcal{F}_n.$ For $\alpha\leq \beta$ we have
\begin{eqnarray*}
\|v_\beta-v_\alpha\|_2 & =& \|l(v_\beta)v_\beta-l(v_\alpha)v_\beta\|_2= \\
& = & \|(l(v_\beta)-l(v_\alpha))v_\beta\|_2\leq
\|l(v_\beta)-l(v_\alpha)\|_2\|v_\beta\|= \\
& = & \sqrt{\tau(l(v_\beta)-l(v_\alpha))}\rightarrow 0,
\end{eqnarray*}
because $\{l(v_\alpha)\}$ is an increasing net of projections.
Thus $\{v_\alpha\}$ is a $\|\cdot\|_2$-fundamental, and hence
there exists an element $v$ in the unit ball $M$ such that
$v_\alpha\stackrel{\|\cdot\|_2}\longrightarrow v.$ Therefore
$v_\alpha\stackrel{t_\tau}\longrightarrow v,$ and thus we have
$$
v_\alpha  v_\alpha^\ast \stackrel{t_\tau}\longrightarrow v
v^\ast,\, v_\alpha^\ast v_\alpha\stackrel{t_\tau}\longrightarrow
v^\ast  v.
$$
 Therefore
$$
v v^\ast,  \, v^\ast v\in P(M).
$$
Thus  $v\in \mathcal{GN}(M).$

Since  $\{v_\alpha v_\alpha^\ast\}$ is an increasing net of
projections it follows that $v_\alpha  v_\alpha^\ast \uparrow v
v^\ast.$ Also, $v_\alpha=v_\alpha v_\alpha^\ast v_\beta$ for all
$\beta\geq \alpha$ implies that $v_\alpha=v_\alpha v_\alpha^\ast
v.$ So $v_\alpha\leq_1 v$ for all $\alpha.$ Since
$v_\alpha\stackrel{t_\tau}\longrightarrow v$ by  ${t_\tau}$-continuity of $D$ we have that
$D(v_\alpha)\stackrel{t_\tau}\longrightarrow D(v).$ Taking into
account that $v_\alpha v_\alpha^\ast D(v_\alpha)v_\alpha^\ast\geq
n v_\alpha v_\alpha^\ast$ we obtain $v v^{\ast} D(v)v^{\ast}\geq n
v v^{\ast},$ i.e. $v\in \mathcal{F}_n.$

So, any totally ordered net  in $\mathcal{F}_n$ has the least upper
bound. By Zorn`s Lemma $\mathcal{F}_n$ has a maximal element, say
$v_n.$

Put
$$
p_n=\mathbf{1}- v_nv_n^\ast \vee v_n^\ast v_n \vee
s(iD(v_nv_n^\ast))\vee s(iD(v_n^\ast v_n)).
$$

Let us prove that
$$
\|vv^\ast D(v)v^\ast\|\leq n
$$
for all $v\in \mathcal{U}(p_nMp_n).$

The case  $p_n=0$  is trivial.

Let us consider the case $p_n\neq 0.$ Take $v\in
\mathcal{U}(p_nMp_n).$ Let $v v^\ast
D(v)v^\ast=\int\limits_{-\infty}^{+\infty}\lambda \, d\,
e_{\lambda}$ be the spectral resolution of $v v^\ast D(v)v^\ast.$
 Assume that $p=e_n^\perp\neq 0.$ Then
$$
p v v^\ast D(v)v^\ast p\geq n p.
$$
Denote $u=pv.$ Then since $p\leq p_n = vv^\ast,$  we have
\begin{eqnarray*}
u u^\ast D(u)u^\ast  & =& p v v^\ast p D(pv)v^\ast p= \\
& = & p v v^\ast p D(p)vv^\ast p +p v v^\ast p p D(v)v^\ast p=\\
& = & p v v^\ast p D(p) p vv^\ast  +p v v^\ast D(v)v^\ast p =\\
& = &0+ p v v^\ast D(v)v^\ast p\geq n p,
\end{eqnarray*}
i.e.
$$
uu^\ast  D(u)u^\ast \geq n p.
$$
Since $u u^\ast, u^\ast u\leq p_n=\mathbf{1}- v_nv_n^\ast \vee
v_n^\ast v_n \vee s(iD(v_nv_n^\ast))\vee s(iD(v_n^\ast v_n))$ it
follows that $u$ is orthogonal to $v_n,$ i.e.
$uv_n^{\ast}=v_n^{\ast}u=0.$ Therefore  $w=v_n+u\in
\mathcal{GN}(M).$ Using  Lemma~\ref{orto} we have
\begin{eqnarray*}
w w^\ast D(w)w^\ast & =& v_n v_n^{\ast}  D(v_n)v_n^{\ast} +u
u^\ast D(u)u^\ast\geq n(v_nv_n^\ast+p)=nww^\ast,
\end{eqnarray*}
because
\begin{eqnarray*}
w w^\ast  & =& (v_n+u)(v_n +u)^{\ast}=v_nv_n^{\ast} +u u^\ast=\\
& =& v_nv_n^\ast+pvv^\ast p=v_nv_n^\ast+p.
\end{eqnarray*}
So
$$
w w^\ast D(w)w^\ast \geq n ww^\ast.
$$
This is contradiction with maximality $v_n^.$ From this
contradiction it follows that $e_n^\perp=0.$ This means that
$$
v v^\ast D(v)v^\ast\leq n v v^\ast
$$
for all $v\in \mathcal{U}(p_nMp_n).$

Set
$$
\mathcal{S}_n=\{v\in \mathcal{GN}(M): vv^\ast D(v)v^\ast\leq -n
vv^\ast\}.
$$
By Lemma~\ref{five} it follows that $v\in \mathcal{F}_n$ is a
maximal element of $\mathcal{F}_n$ with respect to the order
$\leq_1$ if and only if $v^\ast$ is a maximal element of
$\mathcal{S}_n$ with respect to the order $\leq_2.$

Taking into account this observation  in a similar way we can show
that
$$
v v^\ast D(v)v^\ast\geq -n v v^\ast
$$
for all $v\in \mathcal{U}(p_nMp_n).$ So
$$
-n v v^\ast\leq v v^\ast D(v)v^\ast\leq n v v^\ast.
$$
This implies that $v v^\ast D(v)v^\ast\in M$ and
\begin{equation}\label{noreq}
 \|vv^\ast D(v)v^\ast\|\leq n
\end{equation}
for all $v\in \mathcal{U}(p_nMp_n).$

Let us show that the derivation $p_nDp_n$
 maps $p_nMp_n$ into itself. Take $v\in \mathcal{U}(p_nMp_n).$ Then $vv^\ast=v^\ast v=p_n$ and
hence
\begin{eqnarray*}
(p_nDp_n)(v)=p_nD(p_n v p_n)p_n  & =& vv^\ast D(v)v^\ast v\in
p_nMp_n.
\end{eqnarray*}
Since any element from $p_nMp_n$ is a finite linear combination of
unitaries from $\mathcal{U}(p_nMp_n)$ it follows that
\begin{eqnarray*}
p_n D(x)p_n \in p_nMp_n
\end{eqnarray*}
for all $x\in p_nMp_n,$ i.e. the derivation $p_nDp_n$
 maps $p_nMp_n$ into itself.

Let us show that $\tau(v_n v_n^\ast)\rightarrow 0.$
 Let us suppose the opposite, e.g.
there exist a number $\varepsilon>0$ and  a sequence
$n_1<n_2<...<n_k<...$ such that
$$
\tau(v_{n_k} v_{n_k}^\ast)\geq \varepsilon
$$
for all $k\geq1.$  Since $v_{n_k}\in \mathcal{F}_{n_k}$  we have
\begin{equation}\label{inq}
 v_{n_k}v_{n_k}^{\ast} D(v_{n_k})v_{n_k}^{\ast} \geq n_k v_{n_k}v_{n_k}^{\ast}
 \end{equation}
for all $k\geq1.$

Now take an arbitrary number $c>0$ and let $n_k$ be a number such
that $n_k>c\delta,$ where $\delta=\frac{\textstyle
\varepsilon}{\textstyle 2}.$ Suppose that
$$
v_{n_k}v_{n_k}^{\ast} D(v_{n_k})v_{n_k}^{\ast}\in
cV\left(\delta,\delta\right)= V\left(c\delta,\delta\right).
$$
Then there exists a projection $p\in M$ such that
\begin{equation}\label{eps}
||v_{n_k}v_{n_k}^{\ast}  D(v_{n_k})v_{n_k}^{\ast} p||<c\delta,\,
\tau(p^\perp)<\delta.
\end{equation}
Let $v_{n_k}v_{n_k}^{\ast}
D(v_{n_k})v_{n_k}^{\ast}=\int\limits_{-\infty}^{+\infty}\lambda \,
d\, e_{\lambda}$ be the spectral resolution of
$v_{n_k}v_{n_k}^{\ast}  D(v_{n_k})v_{n_k}^{\ast}.$ From
\eqref{eps} using \cite[Lemma 2.2.4]{Mur} we obtain that
$e^\perp_{c\delta}\preceq p^\perp.$ Taking into account
\eqref{inq} we have that $v_{n_k}v_{n_k}^{\ast}\leq
e^\perp_{n_k}.$ Since $n_k>c\delta$ it follows that
$e^\perp_{n_k}\leq e^\perp_{c\delta}.$ So
$$
v_{n_k}v_{n_k}^{\ast} \leq e^\perp_{n_k}\leq
e^\perp_{c\delta}\preceq p^\perp.
$$
Thus
$$
\varepsilon\leq \tau(v_{n_k}v_{n_k}^{\ast} )\leq
\tau(p^\perp)<\delta=\frac{\varepsilon}{2}.
$$
This contradiction implies that
$$
v_{n_k}v_{n_k}^{\ast} D(v_{n_k})v_{n_k}^{\ast}\notin
cV\left(\delta,\delta\right)
$$
for all $n_k>c\delta.$ Since $c>0$ is arbitrary it follows that
the sequence $\{v_{n_k}v_{n_k}^{\ast}
D(v_{n_k})v_{n_k}^{\ast}\}_{k\geq1}$ is unbounded in the measure
topology. Therefore the set $\{vv^\ast D(v)v^\ast: v\in
\mathcal{GN}(M)\}$ is also unbounded in the measure topology.

 On the other hand, the continuity of the derivation $D$
implies that the set $\{xx^\ast D(x)x^\ast: \|x\|\leq 1\}$ is
bounded in the measure topology. In particular,  the set
$\{uu^\ast D(u)u^\ast: u\in \mathcal{GN}(M)\}$ is also bounded in
the measure topology. This contradiction implies that $\tau(v_n
v_n^\ast)\rightarrow 0.$

Finally let us show that
$$
\tau(\mathbf{1}-p_n) \rightarrow 0.
$$

It is clear that
$$
l(iD(v_nv_n^\ast)v_nv_n^\ast)\preceq v_nv_n^\ast,
$$
$$
r(v_nv_n^\ast iD(v_nv_n^\ast))\preceq v_nv_n^\ast.
$$
Since
$$
D(v_nv_n^\ast)=D(v_nv_n^\ast)v_nv_n^\ast+v_nv_n^\ast
D(v_nv_n^\ast)
$$
we have
\begin{eqnarray*}
\tau(s(iD(v_nv_n^\ast))) & =&
\tau(s(iD(v_nv_n^\ast)v_nv_n^\ast+v_nv_n^\ast
iD(v_nv_n^\ast))\leq \\
& \leq & \tau(s(v_nv_n^\ast)\vee l(iD(v_nv_n^\ast)v_nv_n^\ast)\vee
r(v_nv_n^\ast iD(v_nv_n^\ast)))\leq\\
& \leq & \tau(v_n v_n^\ast)+\tau(v_n v_n^\ast)+\tau(v_n
v_n^\ast)=3\tau(v_n v_n^\ast),
\end{eqnarray*}
i.e.
$$
\tau(s(iD(v_nv_n^\ast))) \leq 3\tau(v_n v_n^\ast).
$$
Similarly
$$
\tau(s(iD(v_n^\ast v_n))) \leq 3\tau(v_n^\ast v_n).
$$
Now taking into account that
$$
v_n v_n^\ast \sim v_n^\ast v_n
$$
we obtain
\begin{eqnarray*}
\tau(\mathbf{1}-p_n) & =& \tau(v_nv_n^\ast \vee v_n^\ast v_n \vee
s(iD(v_nv_n^\ast))\vee s(iD(v_n^\ast v_n)))\leq \\
& \leq & \tau(v_n v_n^\ast)+\tau(v_n^\ast v_n)+3\tau(v_n v_n^\ast)+3\tau(v_n^\ast v_n)=\\
& = & 8\tau(v_n v_n^\ast)\rightarrow 0,
\end{eqnarray*}
i.e.
$$
\tau(\mathbf{1}-p_n) \rightarrow 0.
$$
The proof is complete.
\end{proof}

Let $c\in S(M)$ be a central element. It is clear that the mapping
$$
cD:x\rightarrow cD(x),\, x\in S(M)
$$
is a derivation on $S(M).$

\begin{lemma}\label{last}  Let $M$ be a
von Neumann algebra  with a faithful normal finite trace $\tau$,
 $\tau(\mathbf{1})=1.$ There exist an invertible central
element $c\in S(M)$ and a faithful projection $p\in M$  such that
the derivation $cpDp$ maps $pMp$ into itself.
\end{lemma}

\begin{proof}
By Lemma~\ref{boun} there exists a sequence of projections
$\{p_n\}\subset  M$ with $\tau(\mathbf{1}-p_n) \rightarrow 0$ such
that the derivation $p_nDp_n$ maps $p_nMp_n$ into itself for all
$n\in \mathbb{N}.$ By \eqref{noreq} we have
\begin{equation}\label{nnn}
\|v_nv_n^\ast D(v_n)v_n^\ast\|\leq n
\end{equation}
for all $v_n\in \mathcal{U}(p_nMp_n).$

 Let  $z_n=c(p_n)$ be
the central support of $p_n,$ $n\in \mathbb{N}.$ Since $p_n\leq
z_n$ and $\tau(\mathbf{1}-p_n) \rightarrow 0$ we get
$\tau(\mathbf{1}-z_n) \rightarrow 0.$ Thus
$\bigvee\limits_{n\geq1}z_n=\mathbf{1}.$ Set
$$
f_1=z_1,\, f_n=z_n\wedge\left(\vee_{k=1}^{n-1}f_k\right)^\perp,\,
n>1.
$$
Then  $\{f_n\}$ is a sequence of mutually orthogonal central
projections with $\bigvee\limits_{n\geq1}f_n=\mathbf{1}.$

Set
$$
c=\sum\limits_{n=1}^\infty n^{-1}f_n
$$
and
$$
p=\sum\limits_{n=1}^\infty f_np_n,
$$
where convergence of series means the convergence in the strong
operator topology. Then $c$ is an invertible central element in
$S(M)$ and $p$ is a faithful projection in $M.$

Let us show that
$$
\|vv^\ast cD(v)v^\ast\|\leq 1
$$
for all $v\in \mathcal{U}(pMp).$

Take  $v\in \mathcal{U}(pMp)$ and put $v_n=f_n v,$ $n\in
\mathbb{N}.$ Since $f_n\leq z_n$ it follows that $v_n\in
\mathcal{U}(p_nMp_n).$ Taking into account that $f_n c=n^{-1}f_n$
from the inequality \eqref{nnn} we have
$$
\|v_nv_n^\ast cD(v_n)v_n^\ast\|\leq 1.
$$
Notice that
$$
f_nvv^\ast cD(v)v^\ast f_n=(f_nv)(f_nv)^\ast
cD(f_nv)(f_nv)^\ast=v_nv_n^\ast cD(v_n)v_n^\ast.
$$
Since   $\{f_n\}$ is a sequence of mutually orthogonal central
projections we obtain that
$$
||vv^\ast cD(v)v^\ast||=\sup\limits_{n\geq 1}||v_nv_n^\ast
cD(v_n)v_n^\ast||\leq1.
$$
Thus as in the proof of Lemma~\ref{boun} it follows that the
derivation $cpDp$ maps $pMp$ into itself. The proof is complete.
\end{proof}

In the following Lemmata~\ref{lem:eq}-\ref{lem:main} we do  not
assume the  continuity of  derivations.

\begin{lemma}\label{lem:eq}  Let $M$ be an arbitrary von Neumann algebra
  with   mutually
equivalent  orthogonal projections $e, f$  such that
$e+f=\mathbf{1}.$ If $D:S(M)\rightarrow S(M)$ is a derivation such
that $D|_{fS(M)f}\equiv 0$ then
$$
D(x)=ax-xa
$$
for all $x\in S(M),$ where $a=D(u^\ast)u$ and $u$ is a partial
isometry in $M$ such that $u^\ast u=e,\, u u^\ast=f.$
\end{lemma}

\begin{proof}
Since  $D(f)=0$ we have
$$
D(e)=D(\mathbf{1}-f)=D(\mathbf{1})-D(f)=0.
$$
Thus
$$
D(ex)=eD(x),\, D(xe)=D(x)e,
$$
$$
D(fx)=fD(x),\, D(xf)=D(x)f
$$
for all $x\in S(M).$

Now take the partially isometry $u\in M$ such that
$$
u^\ast u=e,\, u u^\ast=f.
$$
Set $a=D(u^\ast)u.$ Since $eu^\ast=u^\ast,$ $u e=u$ we have
$$
a=D(u^\ast)u=D(e u^\ast)u e=e D(u^\ast)u e,
$$
i.e. $a\in e S(M) e=S(eMe).$

We shall show that
$$
D(x)=ax-xa
$$
for all $x\in S(M).$

Consider the following cases.

Case 1. $x=exe.$ Note that
$$
x=exe= u^\ast u x u^\ast u
$$
and
$$
u x u^\ast =f(u x u^\ast)f\in fS(M)f.
$$
Therefore $D(u x u^\ast)=0.$ Further
\begin{eqnarray*}
D(x) & =& D(u^\ast u x u^\ast u)= \\
& = & D(u^\ast)u x u^\ast u+ u^\ast D(u x u^\ast) u + u^\ast u x
u^\ast D(u)=\\
& = & D(u^\ast)u x+  x u^\ast D(u),
\end{eqnarray*}
i.e.
$$
D(x)=D(u^\ast)u x+  x u^\ast D(u).
$$
Taking into account
$$
u^\ast D(u)=D(u^\ast u)-D(u^\ast)u=D(e)-D(u^\ast)u=-a
$$
we obtain
$$
D(x)=ax-xa.
$$

Case 2. $x=e x f. $ Then
\begin{eqnarray*}
D(x) & =& D(exf)
\ = \ D(u^\ast u x u u^\ast)= \\
& = & D(u^\ast)u x u u^\ast + u^\ast D(u x u u^\ast)= ax,
\end{eqnarray*}
because $u x u u^\ast\in f S(M) f$ and $D(u x u u^\ast)=0.$ Thus
$D(x)=ax.$ Since $a\in eS(M)e$ we have
$$
xa=exf eae=0.
$$
Therefore
$$
D(x)=ax-xa.
$$

Case 3. $x=fxe.$ Then
\begin{eqnarray*}
D(x) & =& D(fxe)
\ = \ D(uu^\ast  x  u^\ast u)= \\
& = & D(u  u^\ast  x  u^\ast) u + u u^\ast  x  u^\ast D(u)=x
u^\ast D(u).
\end{eqnarray*}
Since $u^\ast D(u)=-a$ we get  $D(x)=-xa.$ Since $a\in eS(M)e$ have

$$
ax=eae fxe=0.
$$
Therefore
$$
D(x)=ax-xa.
$$

For an arbitrary element $x\in S(M)$ we consider  its
representation of the form $x=exe+exf+fxe+fxf$ and taking into
account the above cases we obtain
$$ D(x)=ax-xa.
$$
The proof is complete. \end{proof}

\begin{lemma}\label{lem:part}  Let $M$ be a  von Neumann algebra
  and let $e, f\in M$ be   projections such that $0\neq f\sim e\leq f^\perp.$ If
$D:S(M)\rightarrow S(M)$ is a derivation with $D|_{fS(M)f}\equiv
0$ then there exists an element $a\in S(M)$ such that
$$
D|_{pS(M)p}\equiv D_a|_{pS(M)p},
$$
 where $p=e+f.$
\end{lemma}

\begin{proof}
Denote $b=D(e)e-eD(e).$ Since $e$ is a projection, one has
$eD(e)e=0.$ Thus
\begin{eqnarray*}
D_b(e) & =& be-eb =\\
& = & \left(D(e)e-eD(e)\right)e-
e\left(D(e)e-eD(e)\right) =\\
& = & D(e)e+eD(e)=D(e^2)=D(e),
\end{eqnarray*}
i.e. $D(e)=D_b(e).$

Now let $x=fxf.$ Taking into account that $ef=0$ we obtain
\begin{eqnarray*}
D_b(x) & =& bx-ba = \\
& = & \left(D(e)e-eD(e)\right)fxf-
fxf\left(D(e)e-eD(e)\right) =\\
& = & -eD(e)fxf-fxfD(e)e=-eD(ef)xf-fxD(fe)e=0,
\end{eqnarray*}
i.e.
 $D_b|_{fS(M)f}\equiv 0.$
Consider the derivation  $\Delta=D-D_b.$ We have
$$
\Delta(p)=(D-D_b)(e+f)=D(e)-D_b(e)=0,
$$
i.e. $\Delta(p)=0.$ Thus
$$
\Delta(pxp)=p\Delta(x)p
$$
for all $x\in S(M).$ This means that $\Delta$ maps $pS(M)p=S(pM
p)$ into itself. So the restriction $\Delta|_{pS(M)p}$ of $\Delta$
on $pS(M)p$ is a derivation. Moreover
$$
\Delta|_{fS(M)f}=(D-D_b)|_{fS(M)f}\equiv 0.
$$
By Lemma~\ref{lem:eq} there exists $c\in pS(M)p$ such that
$$
\Delta|_{pS(M)p}\equiv D_c|_{pS(M)p}.
$$
Then
$$
D|_{pS(M)p}=(\Delta+D_b)|_{pS(M)p}= D_c|_{pS(M)p}+
D_b|_{pS(M)p}=D_{b+c}|_{pS(M)p}.
$$
So
$$
D|_{pS(M)p}=D_{b+c}|_{pS(M)p}.
$$
The proof is complete. \end{proof}

\begin{lemma}\label{lem:two}  Let $M$ be a  von Neumann algebra
of type II$_1$  with faithful normal center-valued trace $\Phi$
  and let $f$ be  a  projection such that $\Phi(f)\geq \varepsilon \mathbf{1},$
  where $0<\varepsilon<1.$  If
$D:S(M)\rightarrow S(M)$ is a derivation such that
$D|_{fS(M)f}\equiv 0$ then $D$ is inner.
\end{lemma}

\begin{proof} Without loss  of generality we may assume that
$\Phi(\mathbf{1})=\mathbf{1}.$ Choose a number $n\in \mathbb{N}$
such that $2^{-n}<\varepsilon.$ Since $M$ is of type II$_1,$ there
exists a projection $f_1\leq f$ such that
$\Phi(f_1)=2^{-n}\mathbf{1}.$ Since $f_1\leq f$ we have
$D|_{f_1S(M)f_1}\equiv 0.$ Therefore  replacing, if necessary,
$f$ by $f_1,$  we may assume that
$\Phi(f)=2^{-n}\mathbf{1}.$

Consider the following cases.

Case 1. $n=1.$ Then $f\sim f^\perp.$ By Lemma~\ref{lem:eq} $D$ is
inner.

Case 2. $n>1.$  Take a projection $e\leq f^\perp$ with $e\sim f.$
Denote $p=e+f.$ Applying Lemma~\ref{lem:part} we can find an
element $a_p\in S(M)$ such that
$$
D|_{pS(M)p}\equiv D_{a_p}|_{pS(M)p}.
$$
Set $\Delta:=D-D_{a_p}.$ Then $\Phi(p)=2^{1-n}\mathbf{1}$ and
$$
\Delta|_{pS(M)p}\equiv 0.
$$
Similarly, applying Lemma~\ref{lem:part} $(n-1)$ times,  we can find
an element $a\in S(M)$ such that $D=D_a.$ The proof is complete.
\end{proof}

\begin{lemma}\label{lem:main}  Let $M$ be a  von Neumann algebra
of type II$_1$  and let $f$ be  a faithful   projection. If
$D:S(M)\rightarrow S(M)$ is a derivation such that
$D|_{fS(M)f}\equiv 0$ then $D$ is inner.
\end{lemma}

\begin{proof} Since $f$ is a faithful,  we see that
$c(\Phi(f))=\mathbf{1},$ where $c(x)=\inf\{z\in P(Z(M)): zx=x\}$
is the central support of the element $x\in S(M).$ There exist a
family $\{z_n\}_{n\in F},$ $F\subseteq\mathbb{N},$  of central
projections from $M$ with $\bigvee\limits_{n\in F} z_n =
\mathbf{1}$ and a sequence $\{\varepsilon_n\}_{n\in F}$ with
$\varepsilon_n> 0$ such that
$$
z_n\Phi(f)\geq \varepsilon_n z_n
$$
for all $n\in F.$ Since $z_n$ is a  central projection, we have
$D(z_n)=0.$ Thus
$$
D(z_n x)=z_nD(x)
$$
for all $x\in S(M).$ This means that $D$ maps $z_nS(M)=S(z_nM)$
into itself. So
 $z_n D|_{S(z_nM)}$  is a
derivation on  $S(z_nM).$ Moreover
$$
z_nD|_{z_nfS(M)z_nf}\equiv 0
$$
and
$$
z_n\Phi(f)\geq \varepsilon_n z_n.
$$
By Lemma~\ref{lem:two} there exists $a_n=z_n a_n\in S(z_nM)$ such
that
$$
z_nD|_{S(z_nM)}\equiv D_{a_n}|_{S(z_nM)}
$$
for all $n\in F.$ There exists a unique element $a\in S(M)$ such
that $z_n a=z_n a_n$ for all $n\in F.$ It is clear that $D=D_a.$
The proof is complete. \end{proof}

\textit{Proof of Theorem~\ref{main}}.  For  finite type I von
Neumann algebras the assertion  has been proved in \cite[Corollary 4.5]{Alb2}.
Therefore it is sufficient to consider the case of type II$_1$ von Neumann
algebras.

Case 1. The trace $\tau$ is finite ( we may suppose without loss of
generality that $\tau(\mathbf{1})=1).$  By Lemma~\ref{last} there
exist an invertible central element $c\in S(M)$ and a faithful
projection $p\in M$  such that
 the derivation  $cpDp$ on  $S(pMp)=pS(M)p$ maps $pMp$ into itself.
 By Sakai's Theorem \cite[Theorem 1]{Sak66} there is an element $a_p\in pMp$ such that
$cpD(x)p=a_p x-xa_p$ for all $x\in pMp.$ Since $cD$ is
$t_\tau$-continuous it follows that
$$
cpD(x)p=a_p x-xa_p
$$ for all
$x\in S(pMp).$ So
$$
 pD(x)p=(c^{-1}a_p) x-x(c^{-1}a_p)
$$ for all
$x\in S(pMp).$

 As in the proof of
Lemma~\ref{lem:part} denote $b=D(p)p-pD(p).$ Then $D(p)=D_{b}(p).$

Consider the derivation $\Delta$ on $S(M)$ defined by
$$
\Delta=D-D_{c^{-1}a_p}-D_{b}.
$$
Then
$$
\Delta(p)=D(p)-D_{c^{-1}a_p}(p)-D_{b}(p)=0,
$$
because $D(p)=D_{b}(p)$ and $c^{-1}a_p\in pMp.$

Let $x\in S(pMp).$ Taking into account that $\Delta(p)=0$ we have
\begin{eqnarray*}
\Delta(x) & =& \Delta(pxp)=p\Delta(pxp)p=\\
& = & pD(pxp)p-pD_{c^{-1}a_p}(pxp)p-pD_{b}(pxp)p =0,
\end{eqnarray*}
because $pD(pxp)p=pD_{c^{-1}a_p}(pxp)p$ and $pbp=0.$ So
$$
\Delta|_{S(pMp)}\equiv 0.
$$
Since $p$ is a faithful projection in $M,$ by Lemma~\ref{lem:main}
$\Delta=D-D_{c^{-1}a_p}-D_{b}$ is an inner derivation. This means
that there exists an element $h\in S(M)$ such that
$$
D=D_{h}+D_{c^{-1}a_p}+D_{b}=D_{h+c^{-1}a_p+b}.
$$

Case 2. Let $\tau$ be an arbitrary  faithful normal semi-finite
trace on $M.$ Take  a family $\{z_i\}_{i\in I}$  of mutually
orthogonal central projections in $M$ with $\bigvee\limits_{i\in
I}z_i=\mathbf{1}$ and such that $\tau(z_i)<+\infty$  for every
$i\in I$ (such family exists because $M$ is a finite algebra).
The map $D_i: S(z_iM)\rightarrow S(z_iM)$ defined by
$$
D_i(x)=z_i D(z_ix),\, x\in S(z_iM)
$$
is a derivation on  $S(z_iM).$ By the case 1 for each $i\in I$ there exists $a_i\in
S(z_iM)$ such that $D_i=D_{a_i}$.  Further there is a
unique element $a\in S(M)$ such that $z_i a=z_ia_i$ for all $i\in
I.$ Now it is clear that $D=D_a.$ The proof is complete.

Recall  that a $\ast$-subalgebra $\mathcal{A}$  of $S(M)$ is
called absolutely solid if  from $x\in S(M),$  $y\in
\mathcal{A},$ and $|x|\leq |y|$ it follows that $x\in \mathcal{A}.$ Note
that  $S(M, \tau)$ is an absolutely solid $\ast$-subalgebra in
$S(M).$

The following theorem gives a solution of the mentioned problem  \cite[Problem 3]{AK2} for the algebra $S(M,\tau$) of
all  $\tau$- measurable operators affiliated with $M$.

\begin{theorem}\label{smt}
Let $M$ be a finite   von Neumann algebra with a faithful normal
semi-finite trace $\tau$. Then every $t_\tau$-continuous
derivation $D:S(M, \tau)\rightarrow S(M,\tau)$ is inner.
\end{theorem}

\begin{proof} As  above take  a family $\{z_i\}_{i\in I}$  of mutually
orthogonal central projections in $M$ with $\bigvee\limits_{i\in
I}z_i=\mathbf{1}$ and such that $\tau(z_i)<+\infty$  for every
$i\in I.$ The map $D_i: S(z_iM, \tau_i)\rightarrow S(z_iM,
\tau_i)$ defined by
$$
D_i(x)=z_i D(z_ix),\, x\in S(z_iM, \tau_i)
$$
is a derivation on $S(z_iM, \tau_i)=S(z_iM),$ where
$\tau_i=\tau|_{z_iM},$ $i\in I.$  Note that the restriction of the
topology $t_\tau$ on $S(z_iM, \tau_i)$ coincides with the topology
$t_{\tau_i}.$ Since $\tau(z_i)<+\infty$  we have that the measure
topology $t_{\tau_i}$ on $S(z_iM, \tau_i)$ coincides with the
locally measure topology.  Therefore the derivation $D_i$ is
continuous in the locally measure topology. By Theorem~\ref{main}
for each $i\in I$ there exists $a_i\in S(z_iM)$ such that
$D_i=D_{a_i}$. Now if we take the unique element $a\in S(M)$ such
that $z_i a=z_ia_i$ for all $i\in I$, then we obtain that
$$
z_i D(x)=D(z_i x)=D_i(z_i x)=a_i(z_ix)-(z_i x)a_i=z_i(ax-xa),
$$
i.e.
$$
 D(x)=ax-xa
$$
for all $x\in S(M, \tau),$ i.e the derivation $D$ is implemented by the
element $a\in S(M)$. Since $S(M, \tau)$ is an absolutely
solid $\ast$-subalgebra in $S(M),$ applying \cite[Proposition
5.17]{BPS} we may choose the element $a$, implementing $D$, from the algebra $S(M, \tau)$ itself.
So $D$ is an inner derivation on $S(M, \tau)$  . The proof is complete.
\end{proof}

\section*{Acknowledgments}

The authors are indebted to the referee for valuable comments and
suggestions.


\bigskip

\begin{thebibliography}{99}





\bibitem{Alb2} S. Albeverio,  Sh.A. Ayupov, K.K. Kudaybergenov,
Structure of derivations on various algebras of measurable
operators for type I von Neumann algebras, J.\ Funct.\ Anal.\
\textbf{256} (2009), 2917-2943.

\bibitem{AK2}
Sh.A. Ayupov, K.K. Kudaybergenov, Derivations on algebras of
measurable operators, Infin.\ Dimens.\ Anal.\ Quantum Probab.\
Relat.\ Top.\  \textbf{13} (2010), 305-337.

\bibitem{AK1}
Sh. A. Ayupov, K. K. Kudaybergenov, Additive derivations on
algebras of measurable operators, J.\ Operator Theory, \textbf{67}
(2012) 495--510.


\bibitem{Ber}{A.~F. Ber, V.~I. Chilin, F.~A. Sukochev,} Non-trivial derivation on
commutative regular algebras,  Extracta Math., \textbf{21} (2006)
107--147.


\bibitem{BPS} A. F. Ber, B. de Pagter,  F. A. Sukochev,
Derivations in algebras of operator-valued functions, J. Operator
Theory, \textbf{66} (2011),  261-300.


\bibitem{Ber2} A. F. Ber, Continuity of derivations on properly infinite
$\ast$-algebras of $\tau$-measurable operators,  Math. Notes,
\textbf{90} (2011), 758-762.





\bibitem{Ber4} {A. F. Ber, V. I. Chilin,  F. A. Sukochev,}
Continuity of derivations in algebras of locally measurable
operators, arXiv:1108.5427.

\bibitem{Ber3} {A. F. Ber, V. I. Chilin, G. B. Levitina, F. A. Sukochev,}
Derivations on symmetric quasi-Banach ideals of compact operators,
J. Math. Anal. Appl., \textbf{397} (2013) 628--643.



\bibitem{Her} I.N. Herstein,
\textit{Noncommutative rings}, The Carus Math.\ Monographs
\textbf{15}, Math.\ Assoc.\ of Amer.\ (1971).


\bibitem{Kad}
R. V. Kadison, Derivations of operator algebras, Ann.\ of Math.\
\textbf{83} (1966), 280-293.



\bibitem{Che}
P. Chernoff, Representation, automorphisms and derivations on some
operators algebras,  J.\ Funct.\ Anal.\ \textbf{12} (1973)
275-289.


 \bibitem{Mur}  {M.~A.~Muratov,   V.~I.~Chilin,}
 \textit{Algebras of measurable and locally measurable
 operators,}
 Institute of Mathematics Ukrainian Academy of Sciences  2007.



\bibitem{Nel}{E.~Nelson,} Notes on non-commutative integration, J. Funct.
Anal. \textbf{15} (1974)  103--116.



\bibitem{Sak66}
S. Sakai, Derivations of $W^*$-algebras, Ann.\ of Math.\
\textbf{83} (1966), 273-279.

\bibitem{Sak68}
S. Sakai, Derivations of simple $C^\ast$-algebras,  J.\ Funct.\
Anal.\ \textbf{2} (1968) 202-206.


\bibitem{Seg}  I. A. Segal,
A non-commutative extension of abstract integration, Ann. of\
Math. \ \textbf{57} (1953), 401--457.



\end{thebibliography}
\end{document}